\documentclass[11pt,a4paper]{amsart}
\usepackage[english]{babel}
\usepackage{microtype}
\usepackage{amsmath}
\usepackage{amscd}
\usepackage{amssymb}
\usepackage{amsthm}
\usepackage[table]{xcolor}
\usepackage[matrix, arrow,curve]{xy}
\usepackage{verbatim}

\usepackage{tikz}
\usepackage{tikz-cd}

\usepackage{hyperref}
\usepackage{cleveref}

\newtheorem{lemma}{Lemma}[section]
\newtheorem{prop}[lemma]{Proposition}
\newtheorem{thm}[lemma]{Theorem}
\newtheorem{cor}[lemma]{Corollary}

\theoremstyle{definition}
\newtheorem{defn}[lemma]{Definition}

\newtheorem{rem}[lemma]{Remark}

\newcommand{\Vect}{\mathrm{-Vect}}

\newcommand{\MMN}{\mathcal{MM}_{\rm Nori}}
\newcommand{\MMSA}{\mathcal{MM}_{\SA}}

\newcommand{\Ch}{\mathcal{C}}
\newcommand{\Xh}{\mathcal{X}}

\newcommand{\Yh}{\mathcal{Y}}

\newcommand{\isom}{\cong}

\newcommand{\tensor}{\otimes}
\newcommand{\cone}{\mathrm{Cone}}

\renewcommand{\ker}{\mathrm{Ker}}
\newcommand{\Ind}{\mathrm{Ind}}

\newcommand{\eff}{\mathrm{eff}}

\newcommand{\Qmod}{\Q\mathrm{-mod}}

\newcommand{\dR}{\mathrm{dR}}
\newcommand{\hodge}{\ul{H}_{\mathcal{H}}}
\newcommand{\sa}{\mathrm{sa}}
\newcommand{\sk}{\mathrm{sk}}

\newcommand{\VValg}{(\Qbar,\Q)\Vect}

\newcommand{\sing}{\mathrm{sing}}
\newcommand{\Hsing}{H^\sing}
\newcommand{\Hsingtilde}{\tilde{H}^\sing}

\newcommand{\pairs}{\mathrm{Pairs}}
\newcommand{\SA}{\mathsf{SA}}
\newcommand{\GSA}{\mathsf{GSA}}
\newcommand{\Simp}{\mathsf{Simp}}
\newcommand{\pairseff}{\pairs^\eff}

\newcommand{\Q}{\mathbb{Q}}
\newcommand{\Qbar}{{\overline{\Q}}}
\newcommand{\Qtilde}{{\widetilde{\Q}}}

\newcommand{\R}{\mathbb{R}}
\newcommand{\C}{\mathbb{C}}
\newcommand{\Pe}{\mathbb{P}} 

\newcommand{\Na}{\mathbb{N}}
\newcommand{\fT}{\mathfrak{T}}
\newcommand{\ul}[1]{\underline{#1}}

\begin{document}
\title{Semi-algebraic motives}
\author{Annette Huber}
\date{\today}
\address{Annette Huber, Math. Institut, Universit\"at Freiburg, Ernst-Zermelo-Str.~1, 79104 Freiburg, Germany}
\email{annette.huber@math.uni-freiburg.de}

\begin{abstract}
We define a category of motives for semi-algebraic spaces and show that
it is trivial. This implies that there is no good extension of algebraic de Rham cohomology to semi-algebraic spaces compatible with the period isomorphism.  
\end{abstract}
\maketitle


\section*{Introduction}

This note is motivated by the theory of periods. On the one hand, we can define them as the entries of period matrices between singular cohomology and algebraic de Rham cohomology for algebraic varieties defined over a number field. This point of view was initiated by Grothendieck in \cite{grothendieck_66}. It is closely linked to the theory of motives, pure and mixed, see Andr\'e's \cite{andre2}.
On the other hand, a more hands on approach in terms of integrals over semi-algebraic sets was advanced by Kontsevich and Zagier, see \cite{kontsevich_zagier}. The two definitions actually agree, see \cite[Chapter~12]{period-buch}. More recently, the author together with Commelin and Habegger established a similar comparison result also for exponential periods (where factors of the form 
$e^{-f}$ are allowed in the integrand), see \cite{expperI, expPerII}. On the one hand, Kontsevich and Zagier also have given an explicit definition in terms of semi-algebraic geometry. 
On the other hand, Fr\'esan and Jossen show how they can be understood as periods of exponential motives, see \cite{fresan-jossen}. There are some subtelties, but roughly these definitions agree and the numbers appear as volumes of definable sets in a certain o-minimal structure.

This raises the question: are these just coincidences or is there a more conceptual explanation for the sudden appearance of semi-algebraic or definable geometry in the theory of periods and motives? Maybe we can extend algebraic de Rham cohomology to semi-algebraic spaces? It turns out, this is not possible, at least not on this naive level. 
In this note we record the obstruction.

We introduce a theory of motives for semi-algebraic spaces over a fixed real closed field $k$ embedded into $\R$. (See Section~\ref{sec:generalisations} on more general possible settings.) It has the same universal property as Nori's category of motives for algebraic varieties: it is an abelian tensor category, universal for all cohomology theories compatible with singular homology. Hence a putative de Rham cohomology would also factor. We go on to show that the category of semi-algebraic motives is trivial, i.e., equivalent to the category of finite dimensional $\Q$-vector spaces. Its Tannaka dual is the trivial group. This would make all periods algebraic over $k$---but they are not. The same arguments apply to Hodge theory---where probably noone expected an extension anyway.

We still hope that a more careful de Rham theory could exist, but it would need to take into account more than just the structure of semi-algebraic spaces, e.g., analytic structures.

\section{Construction}
Let $k$ be a a real closed field contained in $\R$, $\bar{k}=k(i)$ its algebraic closure.
We are going to work in the category of $k$-semi-algebraic sets $X\subset\R^n$ (for varying $n$)
with morphism continuous and $k$-semi-algebraic. 
Let $\Qmod$ be the category of finite dimensional $\Q$-vector spaces.

We proceed in analogy to Nori's approach to motives of algebraic varieties.

\begin{defn}The quiver $\SA_k$ of \emph{semi-algebraic pairs over $k$}
has as vertices tuples $(X,Y,n)$ where $X$ is a $k$-semi-algebraic set,
$Y\subset X$ a $k$-semi-algebraic subset and $n\in\Na_0$ and edges of the form
\begin{itemize}
\item (functoriality) $f_*:(X,Y,n)\to (X',Y',n)$ for continuous
$k$-semi-algebraic $f:X\to X'$ such that $f(Y)\subset Y'$;
\item (boundary) $\partial: (X,Y,n)\to (Y,Z,n-1)$ for
$X\subset Y\subset Z$ inclusion of $k$-semi-algebraic sets.
\end{itemize}
We denote $\Hsing:\SA_k\to\Qmod$ the representation given by
\[ \Hsing: (X,Y,n)\mapsto H_n^\sing(X,Y;\Q).\]
\end{defn}

Recall that Nori attaches to every representation of a quiver an abelian category characterised by a universal property. See \cite[Chapter~7.1]{period-buch} for a survey of the construction and its properties.

\begin{defn}Let $\MMSA^\eff(k)=\Ch(\SA_k,\Hsing)$ be the diagram category of $\SA_k$ with respect to the representaton $\Hsing$. We factor $\Hsing$ as
\[ \SA_k\xrightarrow{\Hsingtilde}\MMSA^\eff(k)\xrightarrow{f}\Qmod.\]
\end{defn}
By construction, the functor $f$ is faithful and exact. The pair
$(\MMSA^\eff(k),\Hsingtilde)$ satisfies a universal property, see \cite[Theorem~7.1.13]{period-buch}.

\section{Alternative descriptions}
Following Nori's approach, we give an alternative construction of the category via the skeletal filtration. Its existence is a standard fact in semi-algebraic geometry.

Before we can go into this, we need to fix terminology. We follow \cite[Chapter~8]{D:oMin}, which differs from the literature in algebraic topology.
Let $n\in\Na_0$. Let $a_0,\dots,a_n\in k^N$ be affine independent.
The \emph{open $n$-simplex} defined by these vectors is the set
\begin{multline*}
\sigma = (a_0,\ldots,a_n)\\
=\left\{ \sum_{i=0}^n \lambda_i a_i \in\R^N:  \text{for all $i$ we have
$\lambda_i>0$ and }\lambda_0+\cdots + \lambda_n =
1\right\}.
\end{multline*}
The closure of $\sigma$ is the \emph{closed $n$-simplex} $[a_0,\ldots,a_n]$.
As usual, a face of $\sigma$ is a simplex spanned by a non-empty
subset of $\{a_0,\ldots,a_n\}$.

A finite set $K$ of simplices in $\R^N$ is called a \emph{complex} if
for all $\sigma_1,\sigma_2\in K$ the intersection
$\overline{\sigma_1}\cap\overline{\sigma_2}$ is either empty or the
closure of common face $\tau$ of $\sigma_1$ and $\sigma_2$.
Van den Dries's definition does not ask for $\tau$ to lie in $K$. So
the \emph{polyhedron} spanned by $K$
$$|K|  = \bigcup_{\sigma \in K} \sigma$$
may not be a closed subset of $\R^N$.
Note that $|K|$ is $k$-semi-algebraic. 
We call $K$ a \emph{closed complex} if $|K|$ is closed or equivalently,
if for all $\sigma\in K$ and all faces $\tau$ of $\sigma$, we have $\tau\in K$.
In this case, $|K|$ is compact.

If $X$ is a $k$-semi-algebraic set, then a 
 \emph{semi-algebraic triangulation} of $X$ is a pair
$(h,K)$ where $K$ is a complex and where
$h:|K|\rightarrow X$ is a $k$-semi-algebraic  homeomorphism.
We write $\Xh$ for a $k$-semi-algebraic set $X$ together with a choice of
triangulation. Triangulations exist by \cite[Theorem~9.2.1]{BCR} or \cite[Theorem~8.2.9]{D:oMin}.

\begin{defn}
%
The quiver $\Simp_k$ quiver has as  vertices tuples $(\Xh,\Yh,n)$ where
$\Xh$ is a closed complex of dimension $n$ and $\Yh=\sk_{n-1}\Xh$ its $(n-1)$-skeleton, and edges of the form
\begin{itemize}
\item (functoriality) $f_*:(\Xh,\Yh,n)\to (\Xh',\Yh',n)$ for continuous
$k$-semi-algebraic $f:X\to X'$ such that $f(Y)\subset Y'$;\footnote{\emph{sic}, we do not require $f$ to be simplicial.}.
\item (boundary) $\partial: (\Xh,\sk_{n-1}\Xh,n)\to (\sk_{n-1}\Xh,\sk_{n-2}\Xh,n-1)$. 
\end{itemize}


\end{defn}
Note that there is a natural forgetful map of quivers $\Simp_k\to \SA_k$.
%
%
The restriction of $\Hsing$ to 
$\Simp_k$ induces a faithful exact functor
\[ \Ch(\Simp_k,\Hsing)\to \MMSA(k).\]
We want to show that the inclusion is an equivalence.
%
%
The following lemma is the key step.

\begin{lemma}
There is a representations of $\SA_k$ in $\Ch(\Simp_k,\Hsing)$ 
such that the diagram
\[\begin{xy}\xymatrix{
 \Simp_k\ar[d]\ar[r]& \SA_k\ar[d]\ar[dl]\\
  \Ch(\Simp_k,\Hsing)\ar[r]& \Qmod
}\end{xy}\]
commutes up to isomorphism.
\end{lemma}

\begin{proof}
We proceed in two steps. 
The first is analogous to Nori's argument in the case of algebraic varieties, see  \cite[Theorem~9.2.22]{period-buch}, but without the complication by affine covers and \v{C}ech-complexes.

Let $\GSA^+_k$ be the quiver with vertices of the form
$(\Xh,\Yh,n)$ where 
$\Xh$ is a semi-algebraic set of dimension $n$ with a fixed triangulation and
$\Yh$ its $(n-1)$-skeleton. The edges are of the same form as in
$\SA_k$. We do not impose additional compatibility with the triangulation.
In contrast to $\Simp_k$ we do not assume the complex to be closed.
We want to represent $\SA_k$ in $\Ch(\GSA_k^+,\Hsing)$.

Let $\Xh=(X,\fT)$ be a semi-algebraic $X$ with a semi-algebraic triangulation
and put  $X_i=|\sk_i\Xh$.
The boundary maps for  the skeletal filtration define a complex
$C(X,\fT)$
\begin{equation*}\label{eq:good_complex}
 0\to \Hsingtilde_d(X_d,X_{d-1})\to \Hsingtilde_{d-1}(X_{d-1},X_{d-2})\to\dots\to \Hsingtilde_1(X_1,X_0)\to \Hsingtilde(X_0)\to 0
\end{equation*}
in $\MMSA(k)$ whose homology agrees with $\Hsingtilde_*(X)$.
Actually, the complex and hence its homology is in the subcategory
$\Ch(\GSA^+_k,\Hsing)$.

The system of semi-algebraic triangulations
of $X$ is directed because any two triangulations have a common refinement.
Let $\fT_1$ be a semi-algebraic triangulation of $X$ and $\fT_2$ 
 a refinement. Then the skeletal filtration of
$(X,\fT_1)$ is is contained in the skeletal filtration of $(X,\fT_2)$.
The induced map $C(X,\fT_1)\to C(X,\fT_2)$ is a quasi-isomorphism because both
compute singular homology of $X$.
We define
\[ C(X)=\varinjlim_{\fT}C(X,\fT)\]
as the direct limit over all triangulations of $X$. The limit exists
in the ind-category of $\Ch(\GSA_k^+,\Hsing)$. The direct limit is exact, hence $C(X)$ is quasi-isomorphic to each $C(X,\fT)$. In particular, it has
homology in
$\Ch(\GSA_k^+,\Hsing)$. 

Given a continuous semi-algebraic map of semi-algebraic sets
$f:V\to W$ and a semi-algebraic triangulation $\fT_V$ on $V$,
we may choose a semi-algebraic triangulation $\fT_W$ of $W$ such that
$f(\sk_i(\fT_V))\subset\sk_i(\fT_W)$. Hence we get an induced map
\[ f_*:C(V,\fT_V)\to C(W,\fT_W).\]
In turn this induces a map
\[ f_*:C(V)\to C(W).\]
In the particular case $Y\subset X$, we write
\[ C(X,Y)=\cone\left( C(Y)\to C(X)\right).\]
We now can write down the representation. We map a vertex $(X,Y,n)$ of
$\SA_k$ to
\[ (X,Y,n)\mapsto H_n(C(X,Y))\in\Ch(\GSA_k^+,\Hsing).\]
An edge $f:(X,Y,n)\to (X',Y',n)$ is mapped to $H_n(f_*)$ for
\[ f_*:C(X,Y)\to C(X',Y').\]
An edge $\partial:(X,Y,n)\to (Y,Z,n-1)$ for  a triple
$Z\subset Y\subset X$ is mapped to the connecting morphism of the
distinguished triangule in $D^b(\Ind-\Ch(\GSA_k^+,\Hsing))$
\[ C(Y,Z)\to C(X,Z)\to C(X,Y).\]
In all, we have constructed our representation in $\Ch(\GSA_k^+,\Hsing)$, finishing the first step.

Note that $\Simp_k$ is a full subquiver of $\GSA_k^+$. We show that the
induced faithful exact functor $\Ch(\Simp_k,\Hsing)\to \Ch(\GSA_k^+,\Hsing)$
is an equivalence. We do this by constructing an inverse as a map of quivers.

Given a triangulated semi-algebraic set $X$, we write $r(X)$ for the 
closed core of the barycentric subdivison of the triangulation of $X$, see \cite[Section~6.2]{expperI}. By loc. cit. Proposition 6.7 it
is a deformation retract of $X$. Note that by construction $\sk_i r(X)=r(\sk_i X)$, hence $r$ induces a map of quivers
\[ r:\GSA^+_k\to\Simp_k.\]
The representation $\Hsing\circ r$ of $\GSA_k^+$ is naturally isomorphic to
$\Hsing$. Hence $r$ induces a functor
\[ r_*: \Ch(\GSA^+_k,\Hsing)\to \Ch(\Simp_k,\Hsing).\] 

We get our representation of $\SA_k$ in $\Ch(\Simp_k,\Hsing)$ by composing the representation of $\SA_k$ in $\Ch(\GSA_k^+,\Hsing)$ with $r_*$.
\end{proof}
\begin{prop}\label{prop:equiv}
The inclusion 
\[ \Ch(\Simp_k,\Hsing)\subset\MMSA(k)^\eff\]
 is an equivalence of categories.
\end{prop}
\begin{proof}
By the universal property of the diagram category, the representation
$\SA_k\to\Ch(\Simp_k,\Hsing)$ induces a faithful exact functor
\[ \MMSA(k)^\eff\to \Ch(\Simp_k,\Hsing).\]
It is inverse to the inclusion. 
This makes  the categories equivalent.
\end{proof}

\section{Tensor product and comparison with Nori motives}
The same methods as in \cite[Chapter~9]{period-buch} can be used to equip an suitable subquiver of good paris of $\SA_k$ with a weak commutative product structure in the sense of loc. cit. Remark 8.1.6 or, equivalently, the structure of a graded $\tensor$-quiver as in 
\cite[Definition~2.13]{BHP}. We omit the details.


\begin{prop}
The category $\MMSA(k)^\eff$ carries a natural commutative tensor product such that $\Hsingtilde_*$ satisfies the K\"unneth formula, i.e., the forgetful functor
$\MMSA(k)^\eff\to\Qmod$ is a tensor functor.
\end{prop}

Recall that every quasi-projective algebraic variety over $\bar{k}=k(i)$ can be embedded into some big $\R^N$ as a bounded $k$-semi-algebraic subspace, see for example \cite[Lemma~2.6.6]{period-buch}. (The reference is over
$\Qbar$, but the same argument works in our case.) 
We denote this functor $X\mapsto X^\sa$.

The functor  induces a map of quivers
\[ \cdot^\sa:\pairseff_{\bar{k}}\to \SA_k\]
where $\pairseff_{\bar{k}}$ is Nori's quiver of effective pairs. Its vertices
are tuples $(X,Y,n)$ with $X$ a quasi-projective variety over $\bar{k}$, $Y$ a closed subvariety and $n\geq 0$. Edges are given by functoriality and boundary as in the semi-algebraic case (or indeed the other way around).
Recall that the category of effective (homological) Nori motives $\MMN(\bar{k})^\eff$ is defined
as the diagram category $\Ch(\pairs^\eff_{\bar{k}},\Hsing)$.

\begin{rem}This is Nori's original set-up. In \cite{period-buch}, the cohomological version is used instead.
The two points of view are dual to each and lead to equivalent theories of non-effective motives.
The author has come round to the idea that homology is more suitable.
\end{rem}

\begin{prop}
The functor $\cdot^\sa$ induces a faithful exact tensor functor
\[ \MMN(\bar{k})^\eff\to\MMSA(k)^\eff.\]
\end{prop}
\begin{proof}
We have a representation 
\[ \pairs^\eff_{\bar{k}}\to \SA_k\xrightarrow{\Hsingtilde}\MMSA(k)^\eff\]
compatible with $\Hsing$.
The existence of the functor follows from the universal property. 
It is a tensor functor because the construction of the tensor functor (which we did not spell out) is compatible.
\end{proof}

We write $M\mapsto M^\sa$ for this functor on motives.

\section{The main result}

\begin{lemma}
Every object of $\MMSA(k)^\eff$ is a subquotient of an object which is the direct sum of objects of the form $\Hsingtilde_n(\Delta_n,\partial \Delta_n)$ for $n\geq 0$. Here $\Delta_n$ is a closed linear $n$-simplex in some $\R^N$ with vertices in $k$.
\end{lemma}
\begin{proof}
We use the characterisation $\MMSA(k)^\eff=\Ch(\Simp_k,\Hsing)$ of Proposition~\ref{prop:equiv}. By general properties of the diagram category, see \cite[Propostition~7.1.16]{period-buch}, every object is subquotient of an object attached to a vertex of the quiver, i.e., of the
form $\Hsingtilde_n(X,Y)$ for a closed complex $X$ of dimension $n$ and $Y=\sk_n X$. By definition, $X$ is constructed from $Y$ by gluing in
finitely many closed simplices of dimension $n$. Let $X^\Delta$ be the disjoint union of these simplices, $Y^\Delta$ the disjoint union of their boundaries. By excision, $\Hsingtilde_n(X^\Delta,Y^\Delta)\to \Hsing_n(X,Y)$ is an isomorphism.
\end{proof}

\begin{thm}
The category $\MMSA(k)^\eff$ is equivalent to the tensor category of finite dimensional $\Q$-vector spaces. The functor $\cdot^\sa:\MMN(k)^\eff\to\MMSA(k)^\eff$
agrees with singular homology.
\end{thm}
\begin{proof}
The argument is the same as for the homotopy category of simplical complexes in algebraic topology. We review it in order to clarify that the isomorphisms involved in the proof are motivic.

Let $\Delta_{n+1}$ be the (closed) standard $(n+1)$-simplex. Its boundary $\partial \Delta_{n+1}$ is
a simplicial version of the $n$-sphere $S_n$. The embedding
of $S_{n-1}$ into $S_n$ as the equator is replaced by the embedding of
$\partial\Delta_{n}$ into $\partial\Delta_{n+1}$ induced by the embedding
of $\Delta_n$ into $\Delta_{n+1}$ as one face. We think of this face as the lower hemisphere. Let $C_n$ be the union of the remaining faces. It is a cone over $\partial \Delta_{n-1}$ and homotopy equivalent to a subdivision of an $n$-simplex. We think of it as the upper hemisphere.

The long exact sequence in homology gives for $n\geq 1$
\[0\to\Hsingtilde_{n}(\partial\Delta_{n+1})\to  \Hsingtilde_{n}(\partial \Delta_{n+1},\partial\Delta_{n})\to
\Hsingtilde_{n-1}(\partial\Delta_{n})\]
(with the last map surjective for $n\geq 2$).
Moreover, the cover $\partial \Delta_{n+1}=\Delta_n\cup C_n$ induces by excision
\begin{align*}
 \Hsingtilde_{n}(\partial \Delta_{n+1},\partial\Delta_n)&\isom
   \Hsingtilde_n(\Delta_n,\partial\Delta_n)\oplus \Hsingtilde_n(C_n,\partial\Delta_n)\\
&\isom \Hsingtilde_n(\Delta_n,\partial \Delta_n)^2.
\end{align*}
Finally, we have for $n\geq 1$
\[ \Hsingtilde_{n+1}(\Delta_{n+1},\partial\Delta_{n+1})\isom \Hsingtilde_n(\partial \Delta_{n+1})\]
via the connecting morphism. For $n=0$, we get
\[ \Hsingtilde_1(\Delta_1,\partial\Delta_1)\isom\ker(\Hsingtilde_0(\partial \Delta_1)\to\Hsingtilde_0(\Delta_1))\isom\Q(0)\]
where we write $\Q(0)$ for the unit object of the tensor structure.
Putting these identifcations together, we get inductively isomorphisms between all the
$\Hsingtilde_n(\Delta_n,\partial\Delta_n)$. They are all isomorphic to $\Q(0)$.
Together with the previous Proposition, this means that every object of
$\MMSA(k)^\eff$ is a subquotient of some $\Q(0)^n$.

The functor $\MMSA(k)^\eff\to \Qmod$ is faithful and exact, hence it suffices to prove fullness. Fullness follows in general if it holds for
$\Q(0)\mapsto \Q$. This case is trivial. 

The functor $\cdot^\sa$ amounts to computing simplicial homology of a variety. This is the same as singular homology.
\end{proof}

In particular, the tensor category $\MMSA(k)^\eff$ is rigid. We drop the
supscript $\eff$ from now on.

\begin{cor}The functor $\cdot^\sa$ extends to a tensor functor between Tannakian categories
\[ \cdot^\sa:\MMN(k)\to \MMSA(k).\]
\end{cor}


\section{Consequences}

We fix an embedding $\Qbar\subset\C$. Let $\Qtilde=\Qbar\cap\R$.

\begin{thm}There is no extension of algebraic de Rham cohomology from
algebraic varieties over $\Qbar$ (i.e., the quiver $\pairseff_\Qbar$)
to semi-algebraic varieties over $\Qtilde$ (i.e., the quiver
$\SA_\Qtilde$) compatible with the period isomorphism between singular cohomology and de Rham cohomology after extension of coefficients to $\C$.
\end{thm}
\begin{proof}
Assume that such an extension exists. This means that the functor
\begin{align*} \ul{H}:\MMN(\Qbar)&\to \VValg\\
                            M&\mapsto (H^\dR(M),H^\sing(M),\mathrm{per})
\end{align*}
(assigning to a motive its de Rham realisation, Betti realisation and the period isomorphm)
factors via $\MMSA(\Qtilde)$. Every object in the latter category is
isomorphic to some $\Q(0)^n$. The periods of $\Q(0)$ are trivial, i.e., in
$\Qbar$. Hence the same is true for all objects in $\MMSA(\Qbar)$ and by the factorisation also for all objects of $\MMN(\Qbar)$.  However, we know that not all periods of motives over $\Qbar$ are themselves algebraic. For example, $2\pi i$ appears as a period of $\Q(1)$.
\end{proof}

\begin{rem}The same argument works for any $k\subset\C$ such that there
is some period not in $\bar{k}$. It does not apply to $\C$ itself. Indeed, we may artificially extend de Rham cohomology to $\SA_\R$ by defining it as singular cohomology with $\C$-coefficients.
\end{rem}

Let $k\subset\C$ be a subfield. We also fix an embedding
$\bar{k}\subset\C$. Let $X$ a variety over $k$. We denote $\hodge^n(X)$ its singular cohomology equipped with Deligne's mixed $\Q$-Hodge structure. The functor extends to $\MMN(k)$.

\begin{thm} Let $k\subset\C$ be a subfield. There is no extension of
the functor assigning to an algebraic variety over $k$ (more precisely, the quiver $\pairseff_k$) its $\Q$-Hodge structure to the category of semi-algebraic sets
over $\bar{k}\cap\R$ (more precisely, to the quiver $\SA_{\bar{k}\cap\R}$).
\end{thm}
\begin{proof}Assume such an extension exists.
As in the proof of the last theorem, this implies that $\hodge(M)$ is isomorphic to $\Q(0)^n$ for all objects of $\MMN(k)$. However, we know that $\hodge(M)$ is is non-trivial for most motives. For example,
$\hodge^2(\Pe^1)=\Q(-1)\neq\Q(0)$.
\end{proof}

\section{Generalisations}\label{sec:generalisations}
\begin{rem} We settled on $\Q$-coefficients in our presentation. This is not necssary. The same arguments work integrally or with coefficients in  any noetherian ring (of homological dimension at most $2$ for the tensor structure.)
\end{rem}
\begin{rem}
We focussed on semi-algebraic sets defined over fields that embedd into
$\R$ because this is the most important case. 
Everything works without any changes for any real closed field when replacing ordinary singular homology by the version for semi-algebraic sets introduced by
Delfs and Knebusch, see \cite{DK-homology}.
The
category of semi-algebraic motives can still be defined and is equivalent to
$\Qmod$.

 Amusingly the representation $\pairseff_{\bar{k}}\to\SA_k$ can now be used to define singular homology for such ground fields without an embedding into $\R$. 
Recall that the standard method for defining singular homology of algebraic varieties is by choosing models over fields $K$ of finite type over $\Q$. Such a $K$ can be embedded into $\C$. This approach is tedious and it is as tedious to verify the fact that the two methods define the same functor, up to isomorphism.
\end{rem}
\begin{rem}
The constructions can also be used to define a category of motives
for definable spaces in any o-minimal structure over any real closed field,
 see \cite{D:oMin}.
Again the resulting tensor category is equivalent to $\Qmod$ by the same arguments.
\end{rem}

\begin{rem}All arguments hinge on the existence of semi-algebraic or more generally definable triangulations. These exist in the $C^0$-setting (this seems standard), but also in the $C^1$-setting, see \cite{ohmoto-shiota-triangulation} and \cite{omin-triang}. However, the questions of $C^p$ or even $C^\infty$-triangulations seems to be open. Note that the standard versions e.g. in \cite{D:oMin} give a $C^0$-triangulation which is $C^p$ on the interior of each face---this is very different from a $C^p$-map.
Hence our arguments do not apply to the category of  semi-algebraic $C^p$-manifolds. Neither do they apply to definable complex analytic spaces as studied by Bakker, Brunebarbe and Tsimerman in
\cite{omingaga} 
\end{rem}

\bibliographystyle{alpha}
\bibliography{periods}

\end{document}